\documentclass[12pt,a4paper]{article}

\usepackage{amsmath}
\usepackage{amssymb}
\usepackage{amsthm}
\usepackage{latexsym}
\usepackage[dvipdfmx]{graphicx}
\usepackage{bm}
\usepackage[dvips]{color}

\theoremstyle{plain}
\newtheorem{theorem}{Theorem}

\newtheorem{lemma}{Lemma}
\newtheorem*{claim}{Claim}

\theoremstyle{definition}
\newtheorem{remark}{Remark}

\newcommand{\RR}{{\mathbb{R}}}
\newcommand{\ZZ}{{\mathbb{Z}}}
\newcommand{\dom}{{\rm dom\,}}
\newcommand{\suppp}{{\rm supp}\sp{+}}
\newcommand{\suppm}{{\rm supp}\sp{-}}
\DeclareMathOperator*{\argmin}{arg\,min}

\newcommand{\DetEq}{\mbox{(\textsc{LP})}}
\newcommand{\Master}{\mbox{(\textsc{Master})}}
\newcommand{\RecS}{(\mbox{\textsc{Sub}}(S))}
\newcommand{\RecT}{(\mbox{\textsc{Sub}}(T))}

\numberwithin{equation}{section}
\usepackage{geometry}
\title{Relationship of Two Formulations for \\ Shortest Bibranchings}

\author{Kazuo Murota\thanks{Department of Economics and Business Administration, Tokyo Metropolitan University, Tokyo 192-0397, Japan. 
{\tt murota@tmu.ac.jp}} 
\and 
Kenjiro Takazawa\thanks{Department of Industrial and Systems Engineering, Faculty of Science and Engineering, Hosei University, Tokyo 184-8584, Japan.  
{\tt takazawa@hosei.ac.jp}} 
}

\date{June 2017 / July 2018}
\begin{document}

\maketitle

\begin{abstract}
The shortest bibranching problem is a common generalization of 
the minimum-weight edge cover problem in bipartite graphs 
and 
the minimum-weight arborescence problem in directed graphs. 
For the shortest bibranching problem, 
an efficient primal-dual algorithm is given by Keijsper and Pendavingh~(1998), 
and 
the tractability of the problem is ascribed to  
total dual integrality in a linear programming formulation 
by Schrijver (1982). 
Another view on the tractability of this problem is afforded by 
a valuated matroid intersection formulation by Takazawa (2012). 
In the present paper, 
we discuss the relationship between these two formulations 
for the shortest bibranching problem. 
We first demonstrate 
that the valuated matroid intersection formulation 
can be derived from the linear programming formulation
through the Benders decomposition,
where integrality is preserved in the decomposition process
and the resulting convex programming is endowed with discrete convexity.
We then show how a pair of primal and dual optimal solutions of one formulation 
is constructed from that of the other formulation, 
thereby providing a connection between polyhedral combinatorics and discrete convex analysis. 

\medskip

\end{abstract}

\section{Introduction}

The {shortest bibranching problem}, 
introduced in \cite{Sch82} (see also \cite{Sch03}), 
is a common generalization of 
the minimum-weight edge cover problem in bipartite graphs 
and 
the minimum-weight arborescence problem in directed graphs. 
In a directed graph $D=(V,A)$ with vertex set $V$ and arc set $A$,
an arc subset $B \subseteq A$ is called a \emph{branching} 
if 
$B$ does not contain a directed cycle and 
every vertex $v$ has at most one arc in $B$ entering $v$. 
For a vertex $r\in V$, 
a  branching $B$ is called an \emph{$r$-arborescence}
if every vertex $v \in V \setminus \{r\}$ has an arc in $B$ entering $v$.
In an undirected graph $G=(V,E)$ with vertex set $V$ and edge set $E$, 
an edge subset $F \subseteq E$ is an \emph{edge cover} if 
the union of the end vertices of the edges in $F$ is equal to $V$.

The \emph{shortest bibranching problem} is described as follows.
Let $D=(V,A)$ be a directed graph $D=(V,A)$, and  
$\{S,T\}$ be a (nontrivial) partition of the vertex set $V$,
that is, $S$ and $T$ are nonempty disjoint subsets of $V$ such that $S \cup T = V$.
A subset $B \subseteq A$ of arcs is called an \emph{$S$-$T$ bibranching} if, 
in the subgraph $(V,B)$, 
every vertex in $S$ reaches $T$ 
and every vertex in $T$ is reachable from $S$. 
We denote the set of nonnegative integers by $\ZZ_+$. 

\begin{description}
\item[Instance.]
	A directed graph $(V,A)$, 
	a partition $\{S,T\}$ of $V$, 
	and 
	a nonnegative integer arc-weight $w \in \ZZ_+\sp{A}$. 
\item[Objective.]
	Find an $S$-$T$ bibranching $B$ minimizing $w(B) = \sum_{a \in B}w(a)$. 
\end{description}

We denote an arc leaving $u$ and entering $v$ by $uv$. 
We also denote 
$A[S] = \{ uv \in A \colon u,v \in S \}$, 
$A[T] = \{ uv \in A \colon u,v \in T \}$, 
and 
$A[S,T] = \{ uv \in A \colon u \in S, v\in T \}$. 
Throughout this paper,
we assume, without loss of generality, 
that there is no arc $uv$ with $u \in T$ and $v\in S$,
which implies that $A = A[S] \cup A[T] \cup A[S,T]$.

The shortest $S$-$T$ bibranching problem includes,
 as special cases,
the minimum-weight edge cover problem in bipartite graphs 
and the minimum-weight $r$-arborescence problem in directed graphs.
If $A[S] = A[T] = \emptyset$, then 
$D=(V,A)$ is a bipartite graph with color classes $S$ and $T$,
and an $S$-$T$ bibranching 
corresponds exactly to an edge cover in 
this bipartite graph (the underlying undirected bipartite graph, to be  
more precise).
If $S = \{r\}$, 
 an inclusion-wise minimal $S$-$T$ bibranching is exactly an $r$-arborescence, and hence
the minimum-weight $r$-arborescence problem is reduced to the shortest $S$-$T$ bibranching problem.

There are several methods to solve 
the shortest bibranching problem in polynomial time. 
First, 
the total dual integrality of a linear programming formulation is proved by Schrijver~\cite{Sch82}, 
and hence the ellipsoid method works. 
Second, based on this formulation, 
a much faster primal-dual algorithm is given by Keijsper and Pendavingh~\cite{KP98}. 
Third, 
the shortest bibranching problem can be described as the shortest strong connector problem in a 
source-sink connected digraph, 
which can be reduced to the weighted matroid intersection problem (see \cite{Sch03} for details). 
Finally, 
a recent work of Takazawa~\cite{Tak12bibr} shows a polynomial reduction of 
the shortest bibranching problem to 
the valuated matroid intersection problem~\cite{Mur96I,Mur96II}, 
and hence any valuated matroid intersection algorithm can solve the shortest bibranching problem.

These results demonstrate that the shortest 
bibranching problem can be 
understood through the standard framework of polyhedral combinatorics~\cite{Sch03}, 
and a relatively new framework of discrete convex analysis~\cite{Mur03} as well. 
In the present paper, 
we discuss the relationship between 
these two approaches to the shortest 
bibranching problem. 
First, 
we demonstrate 
that the valuated matroid intersection formulation 
can be derived from the linear programming formulation
through the Benders decomposition~\cite{BL11,DT97},
where integrality is preserved in the decomposition process
and the resulting convex programming is endowed with discrete convexity.
In this view the valuated matroid intersection formulation corresponds to  
the master problem and the subproblems%
\footnote{These subproblems correspond to  
\emph{recourse problems} in stochastic programming.
} 
are instances of the minimum-weight $r$-arborescence problem. 
This general understanding naturally leads us to 
a solution algorithm analogous to the Bender decomposition. 
The concave functions representing the objective values of the subproblems
are replaced by valuated matroids, which are discrete analogues of concave functions.
Next we discuss the relationship between the two duality theorems associated with 
the linear programming and valuated matroid intersection  
formulations, 
and show how a pair of primal and dual optimal solutions of one formulation 
is constructed from that of the other formulation.

The organization of this paper is as follows. 
In Section~\ref{SECpre}, 
we recapitulate the two formulations for the shortest $S$-$T$ bibranching problem, 
a linear programming formulation and 
a valuated matroid intersection formulation,
where the emphasis is laid on a clear-cut presentation
of the existing derivation of the latter formulation.
In Section~\ref{SECdiscussion}, 
we point out that the valuated matroid intersection formulation 
can also be derived from the linear programming formulation
through the Benders decomposition, 
which turns out to be compatible with integrality and discrete convexity.
In Section~\ref{SECpot}, 
we exhibit how to construct 
a pair of primal and dual optimal solutions for 
the valuated matroid intersection 
formulation from a pair of primal and dual 
optimal solutions 
for the linear programming formulation. 
Section~\ref{SECcut} shows the converse, 
i.e.,\ 
how to construct a pair of primal and dual optimal solutions 
for the linear programming formulation from 
a pair of primal and dual optimal solutions for 
the valuated matroid intersection formulation.

\section{Existing Two Formulations}
\label{SECpre}

\subsection{Linear programming formulation}
\label{SEClp}

In this section, we review the system of linear inequalities 
describing the shortest $S$-$T$ bibranching problem~\cite{Sch82,Sch03}. 
This system of inequalities is a common generalization of 
that for the minimum-weight edge cover problem in bipartite graphs and 
that for the minimum-weight $r$-arborescence problem. 
The total dual integrality of this system
 forms the basis of our understanding of  
the shortest $S$-$T$ bibranching problem in the framework of polyhedral combinatorics~\cite{Sch03}.

Let $D=(V,A)$ be a directed graph, 
$\{S,T\}$ be a (nontrivial) partition of $V$, 
and $w \in \ZZ_+\sp{A}$ be a nonnegative integer arc-weight vector. 
For $X \subseteq V$, 
let 
$\delta\sp{+}X = \{uv \in A \colon \mbox{$u\in X$, $v\in V \setminus X$}\}$ 
and 
$\delta\sp{-}X = \{uv \in A \colon \mbox{$u\in V \setminus X$, $v\in X$}\}$. 
The following linear program~(P) in variable $x \in \RR\sp{A}$ represents the shortest $S$-$T$ bibranching problem: 
\begin{alignat}{3}
\mbox{(P)}\quad{}&{}\mbox{Minimize}     {}&{}&{}\sum_{a \in A} w(a) x(a) {}&{}&{}\notag\\
&{}\mbox{subject to}\quad   
&{}&{}                           \sum_{a\in \delta\sp{+}S'}x(a) \ge 1 \quad{}&{}&{}
(\emptyset \not= S' \subseteq S), \\
&{}{}&{}&{}\sum_{a\in \delta\sp{-}T'}x(a) \ge 1 \quad{}&{}&{}
(\emptyset \not= T' \subseteq T), \\
&{}&{}&{}                           x(a) \ge 0 \quad{}&{}&{}(a \in A).
\end{alignat}
Described below is 
the dual program~(D) of (P), 
whose variables are 
$y \in \RR\sp{2\sp{S} \setminus \{\emptyset\}}$ and $z \in \RR\sp{2\sp{T} \setminus \{\emptyset\}}$:
\begin{alignat}{3}
\mbox{(D)}\quad{}&{}\mbox{Maximize}     {}&{}&{}
\sum_{\emptyset \not= S' \subseteq S} y(S') 
 +  \sum_{\emptyset \not= T' \subseteq T} z(T'){}&{}&{}
\notag
\\
&{}\mbox{subject to}\quad   {}&{}&{}
\sum_{S' \subseteq S, \ a \in \delta\sp{+}S'}y(S') + 
\sum_{T' \subseteq T, \ a \in \delta\sp{-}T'}z(T') \le w(a) \quad{}&{}&{}(a \in A), 
\label{dual1}\\
&{}&{}&{}                           y(S') \ge 0 \quad{}&{}&{}
(\emptyset \not= S' \subseteq S), 
\label{dual2}\\
&{}&{}&{}                           z(T') \ge 0 \quad{}&{}&{}
(\emptyset \not= T' \subseteq T).
\label{dual3}
\end{alignat}
The complementary slackness conditions for (P) and (D) are as follows: 
\begin{align}
\label{EQx} 
&{}x(a) >0 \Longrightarrow \sum_{S' \colon a \in \delta\sp{+}S'}y(S') + \sum_{T' \colon a \in \delta\sp{-}T'}z(T') = w(a), \\
\label{EQy} 
&{}y(S') > 0 \Longrightarrow \sum_{a\in \delta\sp{+}S'}x(a) = 1 ,\\
\label{EQz} 
&{}z(T') > 0 \Longrightarrow \sum_{a\in \delta\sp{-}T'}x(a) = 1 ,
\end{align}
where $a \in A$ in \eqref{EQx},
$\emptyset \not= S' \subseteq S$ in \eqref{EQy},
and
$\emptyset \not= T' \subseteq T$ in \eqref{EQz}.

\begin{theorem}[Schrijver~\cite{Sch82}, see also \cite{Sch03}]
\label{THdicut}
For an arbitrary integer vector $w \in \ZZ\sp{A}_+$, 
{\rm (P)} and {\rm (D)} have integral optimal solutions. 
\end{theorem}

\subsection{M-convex submodular flow formulation}
\label{SECmsbm}

Another formulation of the shortest $S$-$T$ bibranching problem, 
given in \cite{Tak12bibr}, 
falls in the framework of valuated matroid 
intersection~\cite{Mur96I,Mur96II}. 
This formulation provides a new insight into the shortest $S$-$T$ bibranching problem 
through discrete convex analysis~\cite{Mur03}. 
In this paper we adopt
 a formulation by the M$\sp{\natural}$-convex submodular 
flow problem~\cite{Mur99},
which does not differ essentially from 
the valuated matroid intersection formulation~\cite{Tak12bibr},
but offers a clearer 
correspondence to the linear programming formulation in Section~\ref{SEClp}.

We begin with some definitions.
For a finite set $X$ and 
an integer vector $\eta \in \ZZ\sp{X}$, 
we define
$\suppp(\eta) = \{u \in X \colon \eta(u) >0\}$ 
and 
$\suppm(\eta) = \{u \in X \colon \eta(u) <0\}$. 
For $Y \subseteq X$, 
$\chi_{Y}\in \ZZ\sp{X}$ is 
the characteristic vector of $Y$
defined by $\chi_{Y}(u) =1$ if $u \in Y$ 
and $\chi_{Y}(u)=0$ if 
$u \in X \setminus Y$. 
For $u \in X$, 
$\chi_{\{u\}}$ is abbreviated as $\chi_u$. 
For 
a function $f\colon \ZZ\sp{X} \to \overline{\ZZ}$, 
where 
$\overline{\ZZ} =\ZZ \cup \{+\infty\}$, 
\emph{the effective domain} $\dom f$ of $f$ is defined by 
$\dom f = \{\eta \in \ZZ\sp{X}\colon f(\eta) < + \infty\}$. 
A function $f\colon \ZZ\sp{X} \to \overline{\ZZ}$ is called 
an \emph{M$\sp{\natural}$-convex function}~\cite{Mur03,MS99} if 
it satisfies the following exchange property: 
\begin{quote}
For each $\eta,\zeta \in \ZZ\sp{X}$ and 
$u \in \suppp(\eta - \zeta)$, 
it holds that 
\begin{align}
\label{EQmexc1}
f(\eta - \chi_u) + f(\zeta + \chi_u) \le f(\eta) + f(\zeta)
\end{align}
or there exists $v \in \suppm(\eta - \zeta)$ such that 
\begin{align}
\label{EQmexc2}
f(\eta - \chi_u + \chi_v) + f(\zeta + \chi_u - \chi_v) \le f(\eta) + f(\zeta). 
\end{align}
\end{quote}
A set $D \subseteq \ZZ^X$ is called an \emph{M$^\natural$-convex set} if 
its indicator function $\delta_D\colon \ZZ^X \to \RR\cup \{ + \infty\}$ defined by 
\begin{align*}
\delta_D(\eta) = 
	\begin{cases}
	0 			&(\eta \in D), \\
	+\infty 	&(\eta \not\in D)
	\end{cases}
\end{align*}
is an M$^\natural$-convex function. 
Equivalently, 
a set $D \subseteq \ZZ^X$ is an M$^\natural$-convex set if and only if 
it satisfies the following exchange property: 
\begin{quote}
For each $\eta,\zeta \in D$ and 
$u \in \suppp(\eta - \zeta)$, 
it holds that 
\begin{align}
\label{EQmexc1}
\eta - \chi_u \in D \quad \mbox{and} \quad \zeta + \chi_u\in D
\end{align}
or there exists $v \in \suppm(\eta - \zeta)$ such that 
\begin{align}
\label{EQmexc2}
\eta - \chi_u + \chi_v \in D \quad \mbox{and}\quad \zeta + \chi_u - \chi_v \in D. 
\end{align}
\end{quote}

It is pointed out in 
Takazawa~\cite{Tak11ipco,Tak14} that 
discrete convexity 
inherent in 
branchings follows from the arguments in Schrijver~\cite{Sch00}.  
A further connection of $S$-$T$ bibranchings to discrete convex analysis is revealed in \cite{Tak12bibr}. 
In the following,
we summarize the arguments in \cite{Tak11ipco,Tak12bibr,Tak14} 
and exhibit an M$\sp{\natural}$-convex submodular flow formulation 
to highlight
the discrete convexity in the shortest $S$-$T$ bibranching problem.

For the M$\sp{\natural}$-convex submodular flow formulation,
it is convenient to 
regard a (shortest) $S$-$T$ bibranching as a discrete system 
consisting of three components,
a branching, a cobranching, and a bipartite edge cover,
where 
a \emph{cobranching}
means an arc subset such that
the reversal of its arcs is a branching. 
For a precise formulation, we need some notations.

For a digraph $D=(V,A)$ and a partition $\{S,T\}$ of $V$, 
denote the subgraphs induced by $S$ and $T$, respectively, as 
$D[S]$ and $D[T]$, that is, 
$D[S] =(S, A[S])$ and $D[T] = (T, A[T])$.
For $B \subseteq A$, 
denote 
$B[S] = \{uv \in B \colon u,v\in S\}$, 
$B[T] = \{uv \in B \colon u,v\in T\}$, 
and 
$B[S,T] = \{uv \in B \colon \mbox{$u\in S$}, \mbox{$v\in T$}\}$. 
For an arc set $F \subseteq A[S,T]$, 
define 
$\partial\sp{+}F \in \ZZ\sp{S}$ and 
$\partial\sp{-}F \in \ZZ\sp{T}$ by
\begin{align*}
\partial\sp{+}F(u) = |F \cap \delta\sp{+}u| \quad(u \in S), \\
\partial\sp{-}F(v) = |F \cap \delta\sp{-}v| \quad(v \in T), 
\end{align*}
respectively, 
where 
$\delta\sp{+}u = \{uv \in A \colon v\in V \setminus \{u\}\}$ 
and 
$\delta\sp{-}v = \{uv \in A \colon u\in V \setminus \{v\}\}$. 
For a branching $B_T$
in $D[T]$,  let $R(B_T)$ 
denote the set of vertices in $T$ which no arc in $B_T$ enters. 
For a cobranching $B_S$ in $D[S]$, 
let $R\sp{*}(B_S)$ 
denote the set of vertices in $S$ which no arc in $B_S$ leaves.

Then we can say that an arc subset
$B \subseteq A$ is an \emph{$S$-$T$ bibranching} 
if 
$B[S]$ is a cobranching 
with $R\sp{*}(B[S])=\suppp(\partial\sp{+}B[S,T])$ 
and $B[T]$ is a branching with $R(B[T])=\suppp(\partial\sp{-}B[S,T])$. 
Equivalently, 
$B \subseteq A$ is an $S$-$T$ bibranching if 
$B[S]$ is a cobranching in $D[S]$, 
$B[T]$ is a branching in $D[T]$, 
and 
$B[S,T]$ is an edge cover in the graph $D[R\sp{*}(B[S]), R(B[T])]$. 
This definition slightly differs from that in \cite{Sch82}: 
here $B[S]$ should be a cobranching 
and $B[T]$ should be a branching, 
which is not necessarily the case in the definition in \cite{Sch82}. 
However, 
we may 
naturally 
adopt this alternative definition 
as long as we consider the shortest $S$-$T$ bibranching problem.

If we first specify $F \subseteq A[S,T]$ as 
the intersection of $A[S,T]$ and our $S$-$T$ bibranching, 
then arcs in $A[T]$ to be added to $F$ should form a branching $B_T$ in $D[T]$ such that 
$R(B_T) = \suppp(\partial\sp{-}F)$. 
Similarly, a cobranching $B_S \subseteq A[S]$ 
satisfying $R\sp{*}(B_S) = \suppp(\partial\sp{+}F)$ 
should be added to $F$. 
Then an $S$-$T$ bibranching $B$ is obtained as
$B = F \cup B_{S} \cup B_{T}$.
The minimum weights of $B_T$ and $B_S$ are 
expressed respectively by the functions $g_T\colon \ZZ\sp{T} \to \overline{\ZZ}$ 
and $g_S\colon \ZZ\sp{S} \to \overline{\ZZ}$ defined as follows.
The effective domain $\dom g_T$ is defined as 
\begin{align*}
\dom g_T = \{\eta \in \ZZ_+\sp{T}\colon 
\mbox{there is a branching $B_T$ in $D[T]$ with $R(B_T) = \suppp(\eta)$}\}, 
\end{align*}
and, 
for $\eta \in \dom g_T$, 
the function value $g_T(\eta)$ is defined as 
\begin{align}
g_T(\eta) = 
	\min \{ w(B_T) : \mbox{$B_T$ is a branching in $D[T]$, 
        $R(B_T)= \suppp(\eta)$}\} .
\label{EQgTdef}
\end{align}
Similarly, 
we define $g_S\colon \ZZ\sp{S} \to \overline{\ZZ}$ by 
\begin{multline}
\dom g_S = \{\eta \in \ZZ_+\sp{S}\colon 
\mbox{there is a cobranching $B_S$ in $D[S]$ with $R\sp{*}(B_S) = \suppp(\eta)$}\}, \\
\shoveleft
g_S(\eta) = 
	\min \{ w(B_S) : \mbox{$B_S$ is a cobranching in $D[S]$, $R\sp{*}(B_S)= \suppp(\eta)$}\} 
	\\ 
	(\eta \in \dom g_S).  
\label{EQgSdef}
\end{multline}
With $\xi \in \{0,1\}\sp{A[S,T]}$
to represent $F \subseteq A[S,T]$,
the shortest $S$-$T$ bibranching problem is described 
by the following nonlinear optimization problem: 
\begin{align}
\label{EQmsf}
\mbox{(MSF)}
\quad
\mbox{Minimize}\ \  
w(\xi) + g_{S}(\partial \xi|_S) + g_{T}(-\partial \xi|_T)  ,
\end{align}
where 
$w(\xi) = \sum_{a \in A[S,T]} w(a) \xi(a)$, 
and $\partial \xi|_S \in \ZZ\sp{S}$ and $\partial \xi|_T \in \ZZ\sp{T}$ denote the 
restrictions to $S$ and $T$, respectively, of $\partial\xi \in \ZZ\sp{S\cup T}$ 
defined by
\begin{align*}
\partial \xi (v) = |\{ a \colon \xi(a)=1, a \in \delta\sp{+}v \}| 
             - |\{ a \colon \xi(a)=1, a \in \delta\sp{-}v \}| 
\quad (v \in S \cup T).	
\end{align*}

Discrete convexity inherent in the shortest $S$-$T$ bibranching problem
is shown in the following theorem.

\begin{theorem}[Takazawa \cite{Tak12bibr}]
\label{THMconv}
Functions $g_S$ in  \eqref{EQgSdef}
and $g_T$ in \eqref{EQgTdef} 
are M${}\sp{\natural}$-convex functions. 
Thus, the shortest $S$-$T$ bibranching problem is formulated as the 
M$\sp{\natural}$-convex submodular flow problem 
{\rm (MSF)} in \eqref{EQmsf}.
\end{theorem}

We often refer to $\xi \in \{0,1\}\sp{A[S,T]}$ as a \emph{flow}, 
and a flow $\xi$ is said to be \emph{feasible} 
if 
$\partial \xi|_S \in \dom g_S$ 
and 
$-\partial \xi|_T \in \dom g_T$. 
That is, 
$\xi \in \{0,1\}\sp{A[S,T]}$ is feasible if there exist a cobranching $B_S$ in $D[S]$ 
with $R\sp{*}(B_S)=\suppp(\partial\xi|_S)$ and 
a branching $B_T$ in $D[T]$ with $R(B_T)=\suppp(-\partial\xi|_T)$.

\begin{remark}
Note that $\partial \xi$ may not be a $\{0,1\}$-vector,
though $\xi$ itself is a $\{0,1\}$-vector.
Hence 
the domains of $g_S$ and $g_T$ should not be 
restricted to sets of $\{0,1\}$-vectors, 
but 
they are sets of integers.
Therefore, 
in this formulation, 
the framework 
of valuated matroids is not general enough, 
and 
that of M$\sp{\natural}$-convex functions is necessary.  
With some further argument
Takazawa~\cite{Tak12bibr} reduced the formulation~(MSF) to 
the
\emph{valuated matroid intersection problem}~\cite{Mur96I,Mur96II}
so 
that both the original shortest $S$-$T$ bibranching problem and 
the resulting valuated matroid intersection problem 
can be defined on $\{0,1\}$-vectors. 
In this paper, however, we adopt
the M$\sp{\natural}$-convex submodular flow formulation~(MSF) 
in order to make the whole logic clearer. 
\hfill $\Box$
\end{remark}

We now show the proof of Theorem~\ref{THMconv} by clarifying the 
arguments scattered in \cite{Tak11ipco,Tak12bibr,Tak14}. 
The matroidal nature of branchings 
(M$\sp{\natural}$-convexity of $\dom g_T$, to be specific) 
is first noted in \cite{Tak11ipco}. 
For a digraph $D=(V,A)$, 
a \emph{source component} $K$ in $D$ is a strong component such that no arc in $A$ enters $K$, 
where  
we identify a component $K$ and its vertex set 
and denote either of them by $K$. 
It is not difficult to see that, 
for $U \subseteq V$, 
there exists a branching $B$ with $R(B)=U$ 
if and only if 
$U \cap K \neq \emptyset$ for every source component $K$,
where 
$R(B)$ denotes the set of vertices without entering arcs in $B$. 
Hence, 
$\{V \setminus R(B) \colon \mbox{$B$ is a branching in $D$}\}$ 
is an independent set of a partition matroid, 
and 
thus 
$\{\eta \in \ZZ\sp{V} : \mbox{$B$ is a branching in $D$, $R(B)= \suppp(\eta)$}\}$ 
is an M$\sp{\natural}$-convex set (g-matroid).

To prove Theorem \ref{THMconv}, 
we need a stronger exchange property of branchings: 
the arc sets of branchings also have an exchange property. 
First, 
the following lemma is derived from Edmonds' disjoint branchings theorem \cite{Edm73}. 
\begin{lemma}[\cite{Sch00}]
\label{LEMsource}
Let $D=(V,A)$ be a digraph, 
and 
$B_1,B_2$ be branchings partitioning $A$. 
For $R_1',R_2'\subseteq V$ satisfying 
$R_1' \cup R_2' = R(B_1) \cup R(B_2)$ and  
$R_1' \cap R_2' = R(B_1) \cap R(B_2)$, 
the arc set $A$ can be partitioned into branchings $B_1'$ and $B_2'$ 
such that $R(B_1')=R_1'$ and $R(B_2') = R_2'$ 
if and only if 
$K \cap R_1' \neq \emptyset$ and $K \cap R_2' \neq \emptyset$ 
for every source component $K$. 
\end{lemma}

The next lemma, 
which follows from Lemma \ref{LEMsource}, 
describes the exchange property of the arc sets of branchings. 

\begin{lemma}[\cite{Sch00}, see also \cite{Tak12bibr,Tak14}]
\label{LEMexchange}
Let $D=(V,A)$ be a digraph, 
$B_1$ and $B_2$ be branchings partitioning $A$, 
and 
$s \in R(B_1) \setminus R(B_2)$. 
Then, 
there exist branchings $B_1'$ and $B_2'$ which partition $A$ and 
satisfy that 
\begin{itemize}
\item
$R(B_1')= R(B_1) \setminus \{s\}$ and 
$R(B_2')= R(B_2) \cup \{s\}$, 
or 
\item
there exists $t \in R(B_2) \setminus R(B_1)$ such that 
$R(B_1')= (R(B_1) \setminus \{s\}) \cup \{t\}$ and  
$R(B_2')= (R(B_2) \cup \{s\})\setminus \{t\}$. 
\end{itemize}
\end{lemma}

\begin{proof}
Let $K$ be the strong component containing $s$. 
If $K$ is a source component, 
then let $t$ be the root of the directed tree in $B_2$ containing $s$, 
and 
define 
$R_1' = (R(B_1) \cup \{s\}) \setminus \{t\}$ 
and 
$R_2' = (R(B_2) \setminus \{s\}) \cup \{t\}$. 
Note that $t \in K$ and $t \in R(B_2) \setminus R(B_1)$. 
Otherwise, 
define 
$R_1' = R(B_1) \cup \{s\}$ 
and 
$R_2' = R(B_2) \setminus \{s\}$. 
Then the claim follows from Lemma~\ref{LEMsource}. 
\end{proof}

We are now ready to show a proof for Theorem \ref{THMconv}.

\begin{proof}[Proof for Theorem~\ref{THMconv}]
It suffices to deal with $g_T$, 
since the M$\sp{\natural}$-convexity of $g_S$ is proved similarly. 
Let $\eta, \zeta \in \dom g_T$, 
and 
let $u \in \suppp(\eta - \zeta)$. 

If $\zeta(u) \ge 1$, 
then 
$\suppp(\eta - \chi_u) = \suppp(\eta)$ and 
$\suppp(\zeta + \chi_u) = \suppp(\zeta)$, 
which imply 
$g_T(\eta - \chi_u) = g_T(\eta)$ 
and 
$g_T(\zeta + \chi_u) = g_T(\zeta)$. 
Hence  
$g_T(\eta - \chi_u) + g_T(\zeta + \chi_u) \le g_T(\eta)+g_T(\zeta)$ 
in \eqref{EQmexc1}
holds with equality. 

If $\eta(u) \ge 2$ and $\zeta(u)=0$, 
then 
$\suppp(\eta - \chi_u) = \suppp(\eta)$ and 
$\suppp(\zeta + \chi_u) = \suppp(\zeta) \cup \{u\}$, 
which imply 
$g_T(\eta - \chi_u) = g_T(\eta)$ and  
$g_T(\zeta + \chi_u) \le g_T(\zeta)$. 
The latter is derived as follows. 
Let $B_{\zeta}$ be a branching in $D[T]$ yielding $g_T(\zeta)$, 
i.e., 
$R(B_{\zeta})=\suppp(\zeta)$ and $w(B_{\zeta})=g_T(\zeta)$. 
Now $\zeta(u)=0$ implies $u \in T \setminus R(B)$, 
i.e.,  
$B_{\zeta}$ has an arc $a$ entering $u$. 
Then, 
$B'_{\zeta} = B_{\zeta} \setminus \{a\}$ is a branching with $R(B'_{\zeta}) = \suppp(\zeta+ \chi_u)$, 
and thus 
$g_T(\zeta + \chi_u)\le w(B'_{\zeta}) = w(B_{\zeta}) - w(a) \le w(B_{\zeta}) = g_T(\zeta)$, 
where the latter inequality follows from the nonnegativity of $w$.
Therefore 
$g_T(\eta - \chi_u) + g_T(\zeta + \chi_u) \le g_T(\eta)+g_T(\zeta)$ 
in \eqref{EQmexc1}
holds.

If $\eta(u) = 1$ and $\zeta(u)=0$, 
then there exist branchings $B_\eta$ and $B_\zeta$ in $D[T]$ such that 
\begin{align*}
&{}R(B_\eta)=\suppp (\eta), \quad
w(B_\eta)=g_T (\eta), 
\\
&{}R(B_\zeta)=\suppp (\zeta), \quad
w(B_\zeta)=g_T (\zeta) . 
\end{align*}
It is understood that  
in digraph $(T, B_\eta \cup B_\zeta)$, 
an arc $a$ contained in both $B_\eta$ and $B_\zeta$
has multiplicity two in $B_\eta \cup B_\zeta$. 
We have $u \in R(B_\eta) \setminus R(B_\zeta)$. 
By 
Lemma~\ref{LEMexchange}
applied to $(T, B_\eta \cup B_\zeta)$,
there exist branchings $B_\eta'$ and $B_\zeta'$ which partition $B_\eta \cup B_\zeta$ and 
satisfy that 
\begin{align*}
\mbox{$R(B_\eta')= R(B_\eta) \setminus \{u\}$ 
\quad and \quad  
$R(B_\zeta')= R(B_\zeta) \cup \{u\}$} 
\end{align*}
or 
\begin{align*}
\mbox{$R(B_\eta')= (R(B_\eta) \setminus \{u\}) \cup \{v\}$ 
\quad and \quad   
$R(B_\zeta')= (R(B_\zeta) \cup \{u\})\setminus \{v\}$}
\end{align*}
for some $v \in R(B_\zeta) \setminus R(B_\eta)$. 
Then, 
in the former case we obtain 
\begin{align*}
g_T(\eta - \chi_u) + g_T(\zeta + \chi_u) 
\le w(B_\eta') + w(B_\zeta') 
=   w(B_\eta) + w(B_\zeta) 
= g_T(\eta) + g_T(\zeta),
\end{align*}
which shows \eqref{EQmexc1},
and in the latter case, 
\begin{align*}
g_T(\eta - \chi_u + \chi_v) + g_T(\zeta + \chi_u - \chi_v) 
{}&{}\le w(B_\eta') + w(B_\zeta') \\
{}&{}=   w(B_\eta) + w(B_\zeta) 
= g_T(\eta) + g_T(\zeta) ,
\end{align*}
which shows \eqref{EQmexc2}.
This proves M$\sp{\natural}$-convexity of $g_T$. 
\end{proof}

\section{M$\sp{\natural}$-convex Submodular Flow Formulation via Benders Decomposition}
\label{SECdiscussion}

In this section, 
we demonstrate that the M$\sp{\natural}$-convex submodular flow formulation~(MSF) 
can be obtained from the linear program~(P)
through the Benders decomposition,
where integrality is preserved in the decomposition process
and the resulting convex programming is endowed with discrete convexity.

We denote by
$x_{S,T}$, 
$x_{S}$, 
and 
$x_{T}$
the restrictions
$x|_{A[S,T]}$, 
$x|_{A[S]}$, and 
$x|_{A[T]}$ 
of $x$
to
$A[S,T]$, 
$A[S]$, and  
$A[T]$, 
respectively.
Similarly, we use abbreviations
$w_{S,T}=w|_{A[S,T]}$, $w_{S}=w|_{A[S]}$, and $w_{T}=w|_{A[T]}$.
Then the linear program~(P) is rewritten as 
\begin{alignat}{2}
\DetEq\quad
&{}\mbox{Minimize}\quad {}&{}&{}\sum_{a \in A[S,T]}w_{S,T}(a)x_{S,T}(a) 
+ \sum_{a \in A[S]}w_S(a)x_S(a) + \sum_{a \in A[S]}w_T(a)x_T(a) \notag\\
\label{EQst}
&{}\mbox{subject to}\quad{}&{}&{}\sum_{a \in A[S,T]}x_{S,T}(a) \ge 1, \\
&&&{}\sum_{a \in \delta^+ S' \cap A[S,T]}x_{S,T}(a) +  \sum_{a \in \delta^+ S' \cap A[S]}x_{S}(a) \ge 1 
\quad (\emptyset \not= S' \subsetneqq S), \\
&&&{}\sum_{a \in \delta^- T' \cap A[S,T]}x_{S,T}(a) +  \sum_{a \in \delta^- T' \cap A[T]}x_{T}(a) \ge 1 
\quad (\emptyset \not= T' \subsetneqq T), \\
\label{EQnonneg}
&&&{} x_{S,T}, x_S, x_T \ge 0.
\end{alignat}

The Benders decomposition proceeds in the following manner.
The \emph{master problem}, in variable $x_{S,T}$,  is described as
\begin{alignat}{2}
\Master\quad
&{}\mbox{Minimize}\quad {}&{}&{}
\sum_{a \in A[S,T]}w_{S,T}(a)x_{S,T}(a) + h_S(x_{S,T}) + h_T(x_{S,T}) 
\notag\\
\label{EQmaster1}
&{}\mbox{subject to}\quad{}&{}&{}
\sum_{a \in A[S,T]}x_{S,T}(a) \ge 1, \\
\label{EQmaster2}
&&&{} x_{S,T}\ge 0, 
\end{alignat}
where the functions $h_S$ and $h_T$
respectively 
represent the optimal values of the following subproblems
$\RecS$ and $\RecT$
parametrized by $x_{S,T}$:
\begin{alignat}{2}
\RecS\quad
&{}\mbox{Minimize}\quad {}&{}&{}
\sum_{a \in A[S]}w_S(a)x_S(a) \notag\\
&{}\mbox{subject to}
&&{}   \sum_{a \in \delta^+ S'\cap A[S]}x_{S}(a) 
\ge 1 - \sum_{a \in \delta^+ S'\cap A[S,T]} x_{S,T}(a) 
\quad (\emptyset \not= S' \subsetneqq S),\\
&&&{} x_S\ge 0;
\end{alignat}
\begin{alignat}{2}
\RecT\quad
&{}\mbox{Minimize}\quad {}&{}&{}
\sum_{a \in A[T]}w_T(a)x_T(a) \notag\\
\label{EQTcut}
&{}\mbox{subject to}
&&{}   \sum_{a \in \delta^- T' \cap A[T]}x_{T}(a) 
\ge 1 - \sum_{a \in \delta^- T'\cap A[S,T]} x_{S,T}(a) 
\quad (\emptyset \not= T' \subsetneqq T),\\
\label{EQTnonneg}
&&&{} x_T\ge 0.
\end{alignat}
The subproblems $\RecS$ and $\RecT$ 
are linear programs, whereas the master problem $\Master$ 
is a convex program.

We are concerned with a $\{0,1\}$-valued optimal solution 
$x \in \{0,1\}\sp{A}$.
Theorem~\ref{THdicut} guarantees
the existence of an integer optimal solution  for (LP),
and then the constraints~\eqref{EQst}--\eqref{EQnonneg}
imply that it is $\{0,1\}$-valued.
This implies that the master problem $\Master$
and the subproblems $\RecS$ and $\RecT$ 
are also equipped with discreteness.

The combinatorial (or matroidal) nature of
the subproblems can be seen as follows.
Fix $x_{S,T}=\xi \in \{0,1\}\sp{A[S,T]}$
satisfying \eqref{EQmaster1} and \eqref{EQmaster2}.
We first consider $\RecT$.
On noting that \eqref{EQTcut} can be rewritten as
\begin{align*}
 \sum_{a \in \delta^- T'\cap A[T]}x_{T}(a) 
\ge 1 + \partial \xi|_T
\quad (\emptyset \not= T' \subsetneqq T)
\end{align*}
and 
$x_T$ may be assumed to be a $\{0,1\}$-vector,
we can see that 
$\RecT$ is nothing other than the 
problem of finding the minimum-weight branching $B_T \subseteq A[T]$ 
in $D[T]$ with $R(B_T)=\suppp(-\partial \xi|_T)$. 
Thus, 
the optimal value of $\RecT$, denoted $h_T(\xi)$, is in fact equal to 
$g_T(-\partial \xi|_T)$ for the function $g_T$ defined in \eqref{EQgTdef},
i.e., $h_T(\xi) = g_T(-\partial \xi|_T)$. 
In addition, the function $g_T$ is  M${}\sp{\natural}$-convex
by Theorem \ref{THMconv}.
This shows the matroidal property of $\RecT$.
Similarly, we have
$h_S(\xi) = g_S(\partial \xi|_S)$
for the other subproblem $\RecS$,
where
$g_S$ is also an M${}\sp{\natural}$-convex function
by Theorem \ref{THMconv}.

With the above observations the master problem  $\Master$
can be rewritten as:  
\begin{alignat*}{2}
&{}\mbox{Minimize}\quad {}&{}&{}
\sum_{a \in A[S,T]}w_{S,T}(a) \xi(a) 
+ g_S(\partial \xi|_S) + g_T(-\partial \xi|_T) \\
&{}\mbox{subject to}
&&{} 
\xi \in \{0,1\}\sp{A[S,T]},
\end{alignat*}
where
the constraint~\eqref{EQmaster1} in $\Master$ is deleted 
since it is implied by
$\partial \xi|_S \in \dom g_S$ and 
$-\partial \xi|_T \in \dom g_T$. 
Thus,  
the master problem $\Master$ in the Benders decomposition
is equivalent to 
the M$\sp{\natural}$-convex submodular formulation~(MSF) in \eqref{EQmsf}. 

We remark that this observation 
implies that the linear program (P) can be solved by the Benders decomposition, 
in which the subproblems are the minimum-weight $r$-arborescence problem 
and hence can be solved efficiently. 

It is emphasized that the formulation in
the M$\sp{\natural}$-convex submodular problem~(MSF)
in Section \ref{SECmsbm}
is based on purely combinatorial arguments, without directly relying on 
the linear programming formulation (P) or $\DetEq$.
In contrast, 
in this section we have started with the linear programming formulation (P) 
and its integrality (Theorem~\ref{THdicut}),
and derived (MSF) therefrom.

\section{Optimal Flow and Potential from Optimal LP Solutions}
\label{SECpot}

According to the theory of M-convex submodular flows
in discrete convex analysis \cite{Mur99,Mur03},
the M$\sp{\natural}$-convex submodular flow formulation (MSF)
admits an optimality criterion 
in terms of potentials (dual variables).
The objective of this section is to show 
that an optimal flow
and an optimal potential
for (MSF) can be constructed from the optimal solutions 
of the primal-dual pair of linear programs (P) and (D).

The optimality criterion
for M$\sp{\natural}$-convex submodular flows \cite{Mur99,Mur03}, 
when tailored to (MSF), is given in Theorem \ref{THpotexist} below.
For vectors $p \in \ZZ\sp{S}$ and $q \in \ZZ\sp{T}$, 
define functions $g_S[+p]\colon \ZZ\sp{S} \to \overline{\ZZ}$ 
and 
$g_T[+q]\colon \ZZ\sp{T} \to \overline{\ZZ}$ by 
\begin{align*}
&{}g_S[+p](\eta) = g_S(\eta) + \sum_{u\in S}p(u)\eta(u) \quad (\eta \in \ZZ\sp{S}), \\ 
&{}g_T[+q]({\zeta}) = g_T({\zeta}) + \sum_{v\in T}q(v){\zeta}(v) \quad ({\zeta} \in \ZZ\sp{T}),
\end{align*}
where $g_S$ and $g_T$ are given in \eqref{EQgSdef} and \eqref{EQgTdef},
respectively.
\begin{theorem}
\label{THpotexist}
A feasible flow $\xi \in \{0,1\}\sp{A[S,T]}$ is an optimal solution for {\rm (MSF)} 
if and only if 
there exist 
$p \in \ZZ\sp{S}$ 
and 
$q \in \ZZ\sp{T}$
satisfying the following {\rm (i)}--{\rm (iii):} 
\begin{enumerate}
\item for $a =uv \in A[S,T]$,  
\begin{align}  
\xi(a) =1   
&\Longrightarrow \ 
w(a) +  p(u) -  q(v) \leq 0,
\label{EQoptpot1F}
\\
\xi(a) =0
&\Longrightarrow \ 
w(a) +  p(u) -  q(v) \geq 0.
\label{EQoptpot2F}
\end{align}
\item
$\partial\xi|_S \in \argmin( g_{S}[-p])$.
\item
$-\partial\xi|_T \in \argmin( g_{T}[+q])$.
\end{enumerate}
\end{theorem}

We refer to $(p,q) \in \ZZ\sp{S \cup T}$ satisfying (i)--(iii) in Theorem~\ref{THpotexist} 
for some $\xi \in \{0,1\}\sp{A[S,T]}$
as an 
\emph{optimal potential} for (MSF).

We will show how to 
construct an optimal flow $\xi\sp{*} \in \{0,1\}\sp{A[S,T]}$ and 
an optimal potential $(p\sp{*}, q\sp{*}) \in \ZZ\sp{S\cup T}$ for (MSF) 
from the optimal solutions 
 $x \in \{0,1\}\sp{A}$ and 
$(y,z) \in \ZZ\sp{2\sp{S} \setminus \{\emptyset\}} \times \ZZ\sp{2\sp{T} \setminus \{\emptyset\}}$  
of the linear programs (P) and (D).
Recall from Theorem~\ref{THdicut} that both
(P) and (D) have integer optimal solutions.

Given $x$ and $(y,z)$, define
$\xi\sp{*}$ and $(p\sp{*}, q\sp{*})$ by 
\begin{alignat}{2}
	\label{EQxi}
	&{}\xi\sp{*}(a)=x(a)& \quad {}&{}(a \in A[S,T]), \\ 
	\label{EQp}
	&{}p\sp{*}(u) = -\sum_{S' \subseteq S, \ u \in S'}y(S')& \quad {}&{}(u \in S), \\
	\label{EQq}
	&{}q\sp{*}(v) = \sum_{T' \subseteq T, \ v \in T'}z(T')& \quad {}&{}(v \in T). 
\end{alignat}
We prove that $\xi\sp{*}$ and $(p\sp{*},q\sp{*})$ are an optimal flow and an optimal potential for (MSF), 
respectively. 

\begin{theorem}
\label{THpot}
Let
$x \in \{0,1\}\sp{A}$ and 
$(y,z) \in \ZZ\sp{2\sp{S}} \times \ZZ\sp{2\sp{T}}$  
be optimal solutions for {\rm (P)} and {\rm (D)}, 
respectively. 
Then, 
$\xi\sp{*}$ and $(p\sp{*},q\sp{*})$ defined in \eqref{EQxi}--\eqref{EQq} are 
an optimal flow and an optimal potential for {\rm (MSF)}, 
respectively. 
\end{theorem}

\begin{proof}
In the following we 
show (i)--(iii) in Theorem~\ref{THpotexist}.
We first show (i). 
For $a=uv\in A[S,T]$, 
it holds that 
\begin{align*}
	- p\sp{*}(u)+q\sp{*}(v) 
	{}&{}=
	\sum_{S'\subseteq S, \ u \in S'}y(S')
       +\sum_{T'\subseteq T, \ v \in T'}z(T') \\
	{}&{}=  \sum_{S'\subseteq S, \ a \in \delta\sp{+}S'}y(S') 
            + \sum_{T'\subseteq T, \  a \in \delta\sp{-}T'}z(T') \\
	{}&{}\le w(a), 
\end{align*} 
where the last inequality is due to \eqref{dual1}. 
Moreover, if $\xi(a)=1$, 
the inequality turns into an equality by \eqref{EQx}, 
and therefore \eqref{EQoptpot1F} and \eqref{EQoptpot2F} follow. 

Next we show (iii) (rather than (ii)).
Let 
$w'(a) = w(a)-q\sp{*}(v)$ for  
$a = uv \in A[T]$.  
For an arbitrary $\eta \in \dom g_T$, 
it holds that 
	\begin{align}
	g_T[+q\sp{*}](\eta) 
	{}&{}= \min \{w(B) \colon \mbox{$B$ is a branching in $D[T]$, 
               $R(B)=\suppp(\eta)$}\} + \sum_{v \in T} q\sp{*}(v) \eta(v) \notag\\
	{}&{}= \min \{w'(B)  \colon \mbox{$B$ is a branching in $D[T]$, 
               $R(B)=\suppp(\eta)$}\}  \notag\\
	{}&{}\quad + q\sp{*}(T) + \sum_{v \in \suppp(\eta)}q\sp{*}(v)(\eta(v)-1). 
	\label{EQ1}
	\end{align}
A lower bound for the right-hand side of \eqref{EQ1} is provided as follows. 
For the first term we have 
\begin{align}
& \min \{w'(B)  \colon \mbox{$B$ is a branching in $D[T]$, $R(B)=\suppp(\eta)$}\}  
\notag \\
 & \ge \ 
 - \sum_{\emptyset \not= T'\subseteq T}(|T'|-1)z(T'), 
\label{INEQ0}
\end{align}
since, for any branching $B$ in $D[T]$ with  $R(B)=\suppp(\eta)$, it holds that
\begin{align}
w'(B) 	= {}&{}
	\sum_{uv \in B}(w(uv) - q\sp{*}(v))  
\notag\\
	= {}&{} 
      \sum_{uv \in B}\left(w(uv) - \sum_{T'\subseteq T, \  v \in T'}z(T')\right)
\notag\\
	\ge {}&{}
       \sum_{uv \in B}\left(\sum_{T'\subseteq T,\ uv \in \delta\sp{-}T'}z(T') - 
	                         \sum_{T'\subseteq T, \  v \in T'}z(T')\right)  
\label{INEQ1}\\
	= {}&{}
      \sum_{uv \in B}\left(- \sum_{T'\subseteq T,   \   u,v \in T'}z(T')\right)
\notag\\
	= {}&{}
      - \sum_{\emptyset \not= T'\subseteq T}|B[T']|\cdot z(T') 
\notag\\
	\ge {}&{} - \sum_{\emptyset \not= T'\subseteq T}(|T'|-1)z(T'), \label{INEQ2}
	\end{align}
where the first inequality is by \eqref{dual1}.
In addition, 
the last term of the right-hand side of \eqref{EQ1} is nonnegative, 
i.e.,\ 
	\begin{align}
	\sum_{v \in \suppp(\eta)}q\sp{*}(v)(\eta(v)-1) \ge 0, 
	\label{INEQ3}
	\end{align}
since $q\sp{*}(v) \ge 0$ by \eqref{EQq}.
 {} From \eqref{EQ1}, \eqref{INEQ0}, and \eqref{INEQ3}, we  obtain 
\begin{align*}
g_T[+q\sp{*}](\eta) \ge - \sum_{\emptyset \not= T'\subseteq T}(|T'|-1)z(T') + q\sp{*}(T), 
\end{align*}
where the right-hand side is a constant for a fixed $z$. 
Hence, 
in order to prove $-\partial\xi\sp{*}|_{T} \in \argmin g_T[+q\sp{*}]$, 
it suffices to show that 
the three inequalities \eqref{INEQ1}, \eqref{INEQ2}, and \eqref{INEQ3} 
in the above 
turn into equalities 
when $\eta = -\partial\xi\sp{*}|_{T}$. 

For the first and second inequalities~\eqref{INEQ1} and \eqref{INEQ2}, 
let $B\sp{*} = \suppp(x)$ be the shortest $S$-$T$ bibranching corresponding to $x$. 
Then 
$B\sp{*}[T]$ is a branching in $D[T]$ such that $R(B\sp{*}[T])=\suppp(-\partial\xi\sp{*}|_T)$, 
and 
the first inequality~\eqref{INEQ1} holds with equality for $B\sp{*}[T]$ by \eqref{EQx}.
Moreover, 
$|B\sp{*} \cap \delta\sp{-}T'| = 1$ 
for every nonempty $T' \subseteq T$ with $z(T') > 0$ by \eqref{EQz}. 
Thus, 
$|B\sp{*}[T']| = |T'| - |B\sp{*} \cap \delta\sp{-}T'|= |T'|-1$ if $z(T') > 0$, 
and hence the equality in \eqref{INEQ2} follows. 
For the third inequality~\eqref{INEQ3}, 
suppose $q\sp{*}(v)>0$ and 
let $T' \subseteq T$ contribute to $q\sp{*}(v)$ in \eqref{EQq}, 
i.e., 
$v \in T'$ and $z(T') >0$. 
Since $v \in \suppp(-\partial\xi\sp{*}|_{T})$, 
there exists at least one arc 
$a\sp{*} = uv \in A[S,T]$
such that $x(a\sp{*})=1$. 
Then we have that $a\sp{*} \in \delta\sp{-}T'$. 
We also have 
$\sum_{a\in \delta\sp{-}T'}x(a)=1$
by \eqref{EQz}, 
and hence such $a\sp{*}$ is unique. 
Therefore $-\partial\xi\sp{*}|_{T}(v)=1$ follows.
Hence all terms in the summation in \eqref{INEQ3} are equal to zero.

Finally, condition~(ii) is proved similarly to (iii). 
\end{proof}

\section{Optimal LP Solutions from an Optimal Flow and Potential}
\label{SECcut}

In this section, 
we describe how to construct optimal solutions for (P) and (D) 
of the linear programming formulation
from an optimal flow $\xi \in \{0,1\}\sp{A[S,T]}$  and 
an optimal potential $(p,q) \in \ZZ\sp{S \cup T}$ 
for the M$\sp{\natural}$-convex submodular flow formulation (MSF).

We first establish the following lemma, 
in which
$(p,q)$ need not be an optimal potential but an arbitrary pair of vectors. 
\begin{lemma}
\label{lem:deg2}
For arbitrary $p\in \ZZ\sp{S}$ and $q \in \ZZ\sp{T}$, 
the following hold. 
\begin{itemize}
\item 
If $\argmin g_S[-p] \neq \emptyset$, 
then $p(u)\le 0$ for every $u \in S$. 
Moreover, 
$p(u) = 0$ if 
$\eta\sp{*}(u) \geq 2$ 
for some $\eta\sp{*} \in \argmin g_S[-p]$. 
\item 
If $\argmin g_T[+q] \neq \emptyset$, 
then $q(v)\ge 0$ for every $v \in T$. 
Moreover, 
$q(v) = 0$ if 
$\eta\sp{*}(v) \geq 2$ 
for some $\eta\sp{*} \in \argmin g_T[+q]$. 
\end{itemize}
\end{lemma}

\begin{proof}
It suffices to prove the latter assertion. 
Suppose that $q(v) < 0$ for some $v \in T$. 
Note that $\chi_T \in \dom g_T$. 
Then, 
for an arbitrary positive integer $\alpha$, 
we have that 
\begin{align*}
g_T[+q](\chi_T + \alpha \chi_v) 
= g_T(\chi_T) + q(T) + \alpha q(v), 
\end{align*}
which tends to $- \infty$ as $\alpha \to +\infty$. 
Therefore, 
$\argmin g_T[+q] \neq \emptyset$ implies $q(v) \ge 0$ for every $v \in T$. 

Now suppose that 
$\eta\sp{*} \in \argmin g_T[+q]$ and 
$\eta\sp{*}(v) \ge 2$. 
Then 
$\suppp(\eta\sp{*}) = \suppp(\eta\sp{*} - \chi_v)$, 
and hence 
$g_T(\eta\sp{*}) = g_T(\eta\sp{*} - \chi_v)$,
whereas 
$g_T[+q](\eta\sp{*}) \le g_T[+q](\eta\sp{*} - \chi_v)$
by
$\eta\sp{*} \in \argmin g_T[+q]$.
Therefore, we have 
\begin{align*}
g_T[+q](\eta\sp{*}) 
{}&{}\le g_T[+q](\eta\sp{*} - \chi_v) \\
{}&{} = g_T(\eta\sp{*} - \chi_v) + q\cdot(\eta\sp{*} - \chi_v) \\
{}&{} = g_T[+q](\eta\sp{*}) - q(v), 
\end{align*}
which implies $q(v) \le 0$. 
Therefore, 
$q(v) = 0$ follows. 
\end{proof}

We next show the existence of an optimal potential satisfying a property stronger than \eqref{EQoptpot1F}. 

\begin{lemma}
\label{lem:tight}
For an optimal flow $\xi \in \{0,1\}\sp{A[S,T]}$, 
there exists an optimal potential $(p,q) \in \ZZ\sp{S \cup T}$ such that 
\begin{align}
\xi(a) =1   
&\Longrightarrow \ 
w(a) +  p(u) -  q(v) = 0 
\label{EQtight}
\end{align}
holds 
for every $a=uv \in A[S,T]$. 
\end{lemma}

\begin{proof}
Let 
$(p\sp{\circ}, q\sp{\circ})$ be a given optimal potential 
and assume that 
\eqref{EQtight} fails 
for $a\sp{*}=u\sp{*}v\sp{*} \in A[S,T]$.
This means, by \eqref{EQoptpot1F}, that 
$\xi(a\sp{*})=1$ and 
$w(a\sp{*}) + p\sp{\circ}(u\sp{*}) - q\sp{\circ}(v\sp{*}) <0$.

By Lemma~\ref{lem:deg2},  
it holds that 
$p\sp{\circ}(u\sp{*}) \le 0$ 
and $q\sp{\circ}(v\sp{*}) \ge 0$. 
Then, there exist $\alpha,\beta \in \ZZ$ such that 
\begin{align*}
p\sp{\circ}(u\sp{*}) \le \alpha \le 0, 
&&
0 \le \beta \le q\sp{\circ}(v\sp{*}), 
&&
w(a\sp{*}) + \alpha - \beta = 0.
\end{align*}
With such  $\alpha, \beta$
we modify 
$(p\sp{\circ}, q\sp{\circ})$ 
to $(p', q') \in \ZZ\sp{S} \times \ZZ\sp{T}$ as
\begin{align*}
p'(u) =
\begin{cases}
p\sp{\circ}(u)    & (u \in S \setminus \{u\sp{*}\}), \\
\alpha  & (u=u\sp{*}), 
\end{cases}
&&
q'(v) = 
\begin{cases}
q\sp{\circ}(v)   & (v \in T \setminus \{v\sp{*}\}), \\
\beta  & (v=v\sp{*}).
\end{cases}
\end{align*}
Note that \eqref{EQtight} holds 
for $a\sp{*}=u\sp{*}v\sp{*}$ with respect to 
the modified potential $(p', q')$.

\begin{claim}
\label{clm:newpot}
$(p',q')$ is an optimal potential. 
\end{claim}

\begin{proof}[Proof for Claim]
We prove that $\xi$ and 
$(p',q')$ satisfy (i)--(iii) in Theorem~\ref{THpotexist}. 
Note that 
(i)--(iii) in Theorem~\ref{THpotexist} hold for 
$\xi $ and $(p\sp{\circ},q\sp{\circ})$. 

We first show (i). 
Inequality~\eqref{EQoptpot2F} follows from 
$p' \ge p\sp{\circ}$ and 
$q' \le q\sp{\circ}$. 
As for \eqref{EQoptpot1F}, 
it is obvious that \eqref{EQoptpot1F}  holds for $a\sp{*}$. 
Let $\hat{a} = \hat{u}\hat{v} \in A[S,T] \setminus \{a\sp{*}\}$ be such that  
$\xi(\hat{a}) =1$. 
If $\hat{a}$ is not adjacent to $a\sp{*}$, 
then 
$w(\hat{a}) + p'(\hat{u}) - q'(\hat{v}) = w(\hat{a}) + p\sp{\circ}(\hat{u})-q\sp{\circ}(\hat{v}) \le 0$. 
Suppose that $\hat{a}$ is adjacent to $a\sp{*}$,
i.e., $\hat{v} = v\sp{*}$ or $\hat{u} = u\sp{*}$.
If $\hat{v} = v\sp{*}$, then
$-\partial\xi(\hat{v}) \ge 2$, 
and  $q\sp{\circ}(\hat{v})=0$ follows from Lemma~\ref{lem:deg2}. 
Therefore, 
$q'(\hat{v}) = q\sp{\circ}(\hat{v}) = 0$, 
and hence 
\eqref{EQoptpot1F} holds for $\hat{a}$. 
The other case of $\hat{u} = u\sp{*}$ can be treated similarly.

We next show (iii), while noting that (ii) can be proved similarly as (iii). 
Suppose, to the contrary, that 
$-\partial\xi|_{T} \not\in \argmin(g_T[+q'])$. 
That is, 
$g_T[+q'](\eta) < g_T[+q'](-\partial\xi|_{T})$ holds for some $\eta \in \ZZ_+\sp{T}$. 
Here, 
we claim the following:
\begin{align}
-\partial\xi(v\sp{*}) &= 1, \label{EQone}\\
\eta(v\sp{*}) & \ge 2.          \label{EQmore}
\end{align}
\begin{description}
\item[Proof for \eqref{EQone}.] 
Since 
$-\partial\xi|_{T} \in \argmin(g_T[+q\sp{\circ}])$ and $-\partial\xi|_{T} \not\in \argmin(g_T[+q'])$, 
we have that $q' \neq q\sp{\circ}$ and consequently 
$0 \le q'(v\sp{*}) < q\sp{\circ}(v\sp{*})$. 
Then, 
\eqref{EQone} follows from 
Lemma~\ref{lem:deg2}. 
\item[Proof for \eqref{EQmore}.]
Denote $\Delta = q\sp{\circ}(v\sp{*}) - q'(v\sp{*}) >0$.
Since
\begin{align*}
0 &< g_T[+q'](-\partial\xi|_{T}) - g_T[+q'](\eta)
\\
{}&{}
= \bigg( g_T[+q\sp{\circ}](-\partial\xi|_{T}) - g_T[+q\sp{\circ}](\eta)  \bigg)
 + \Delta\cdot \bigg( \eta(v\sp{*}) + \partial\xi(v\sp{*}) \bigg)
\\
{}&{}
\le 
\Delta\cdot \bigg( \eta(v\sp{*}) + \partial\xi(v\sp{*}) \bigg),
\end{align*} 
we have that
$\eta(v\sp{*}) \ge -\partial\xi(v\sp{*})+1 = 2$.
\end{description}

For $\hat{\eta} \in \ZZ_+\sp{T}$ defined by 
\begin{align*}
\hat{\eta}(v)=
\begin{cases}
1          & (v = v\sp{*}),\\
\eta(v)  & (v \in T \setminus\{v\sp{*}\}),
\end{cases}
\end{align*}
it holds that 
\begin{align*}
g_T[+q\sp{\circ}](\hat{\eta}) 
{}&{}=  g_T[+q'](\eta) - q'(v\sp{*}) (\eta(v\sp{*}) - 1) + \Delta \\
{}&{}\le g_T[+q'](\eta) + \Delta \\
{}&{}<   g_T[+q'](-\partial\xi|_{T}) + \Delta \\
{}&{}=   g_T[+q\sp{\circ}](-\partial\xi|_{T}), 
\end{align*}
where \eqref{EQone} and \eqref{EQmore} are used.
This contradicts $-\partial\xi|_{T} \in \argmin (g_T[+q\sp{\circ}]) $. 
Thus 
we have shown
$-\partial\xi|_{T} \in \argmin(g_T[+q'])$
in (iii).
This completes the proof of the claim.
\end{proof}

By the above claim, 
we can reduce the number of arcs violating \eqref{EQtight}
by modifying $(p,q)=(p\sp{\circ},q\sp{\circ})$ to $(p,q)=(p',q')$, 
while maintaining the optimality. 
By repeating such modifications we eventually arrive at the situation where 
\eqref{EQtight} holds for every $a = uv \in A[S,T]$. 
This completes the proof for Lemma~\ref{lem:tight}. 
\end{proof}

In what follows, we assume that
$\xi$ is an optimal flow and 
$(p,q)$ is an optimal potential satisfying the condition~\eqref{EQtight} 
in Lemma~\ref{lem:tight}. 
We construct optimal solutions for (P) and (D)
by considering minimum-weight arborescence problems 
in auxiliary directed graphs
and using well-known results on the linear programming formulation 
of the minimum-weight arborescence problem.

Let $D_T= (V_T, A_T)$ be a directed graph
with arc weight $w' \in \ZZ\sp{A_T}$ defined as follows: 
\begin{align*}
V_T = \{r_T\} \cup T, 
\quad
A_T = \{r_T v \colon v \in T\} \cup A[T], 
\quad
w'(uv) = 
\begin{cases}
q(v)   & (u = r_T), \\
w(uv)  & (u \in T), 
\end{cases}
\end{align*}
where $r_T$ is a newly introduced additional vertex. 
For any $r_T$-arborescence $\tilde{B}_T$ in $D_T$,
$B_T = \tilde{B}_T \cap A[T] = \tilde{B}_T[T]$ is a branching in $D[T]$
with $R(B_T) = \{ v \in T \colon r_T v \in \tilde{B}_T \}$.
Conversely, for 
any branching $B_T$ in $D[T]$,
$\tilde{B}_T = B_T \cup \{ r_T v \colon v \in R(B_T) \}$
is an $r_T$-arborescence in $D_T$.

\begin{lemma}
\label{LEMarboRt}
There exists in $D_T$ 
a minimum-weight $r_T$-arborescence $\tilde{B}_T$ 
such that $R(\tilde{B}_T[T]) = \suppp(-\partial\xi|_{T})$.
\end{lemma}
\begin{proof}
By the correspondence between $r_T$-arborescences in $D_T$
and branchings in $D[T]$ described above, 
the minimum-weight $r_T$-arborescence problem in $D_T$
with respect to $w'$ 
is equivalent to minimizing
$w(B_T) + \sum_{v \in R(B_T)} q(v)$ over branchings $B_T$ in $D[T]$.
On the other hand, in minimizing $g_T[+q](\eta)$, we may assume $\eta \in \{0,1\}\sp{T}$
by Lemma~\ref{lem:deg2},
and for $\eta = \chi_X$ with $X \subseteq T$, 
the value of $g_T[+q](\chi_X)$ is equal 
to the minimum of 
$w(B_T) + \sum_{v \in X} q(v)$
 for a branching $B_T$ in $D[T]$ satisfying $R(B_T) = X$. 
Since $-\partial\xi|_{T} \in \argmin g_T[+q]$, 
there exists 
a minimum-weight branching $B_T$ in $D[T]$ satisfying $R(B_T) = \suppp(-\partial\xi|_{T})$. 
Then the corresponding $r_T$-arborescence
$\tilde{B}_T = B_T \cup \{ r_T v \colon v \in R(B_T) \}$
is a minimum-weight $r_T$-arborescence
such that $R(\tilde{B}_T[T]) = \suppp(-\partial\xi|_{T})$.
\end{proof}

The following problems (P$'$) and (D$'$), 
whose variables are $x' \in \RR\sp{A_T}$ and $\rho \in \RR\sp{2\sp{T}}$, 
are a linear programming formulation of the minimum-weight $r_T$-arborescence problem 
 in $D_T$ and its dual program, 
respectively \cite{Edm67,Sch03}:
\begin{alignat}{3}
\notag
\mbox{(P$'$)}\quad{}&{}\mbox{Minimize}     {}&{}&{}\sum_{a \in A_T} w'(a) x'(a) {}&{}&{}\\
&{}\mbox{subject to}\quad   
&{}&{}\sum_{a\in \delta\sp{-}v}x'(a) = 1 \quad{}&{}&{}(v \in T), \\
&{}{}&{}&{}\sum_{a\in \delta\sp{-}T'}x'(a) \ge 1 \quad{}&{}&{}(T' \subseteq T ,|T'| \ge 2), \\
&{}&{}&{}                           x'(a)\ge 0 \quad{}&{}&{}(a \in A_T).
\end{alignat}
\begin{alignat}{2}
\mbox{(D$'$)}\quad{}&{}\mbox{Maximize}{}&{}&{}
\sum_{v \in T} \rho(v) +  
\sum_{T'\subseteq T, \ |T'|\ge 2}
\rho(T') \notag\\
&{}\mbox{subject to}\quad   {}&{}&{}
\rho(v) + \sum_{T' \colon |T'|\ge 2, \  a \in \delta\sp{-}T'}\rho(T') \le w'(a) \quad(a=uv \in A_T), 
\label{arbdual1}
\\
&{}&{}&{}                           \rho(T') \ge 0 \qquad(T' \subseteq T , |T'| \ge 2).
\label{arbdual2}
\end{alignat}
The complementary slackness conditions for (P$'$) and (D$'$) are as follows:
\begin{align}
\label{EQarbx} 
&{}x'(a) >0 \Longrightarrow 
\rho(v) + \sum_{T' \colon |T'|\ge 2, \  a \in \delta\sp{-}T'}\rho(T')= w'(a), \\
\label{EQarbrho} 
&{}\rho(T') > 0 \Longrightarrow \sum_{a\in \delta\sp{-}T'}x'(a) = 1,
\end{align}
where $a = uv \in A_T$ in \eqref{EQarbx} and
$T' \subseteq T$ with $|T'| \ge 2$ in \eqref{EQarbrho}.

It is known \cite{Edm67,Sch03} that
there exists an integer optimal solution $\rho\sp{*}$ for (D$'$) 
such that
$\rho\sp{*}(v)$ is nonnegative for all $v \in T$, i.e.,
\begin{align}
\label{EQrhovnonneg} 
 \rho\sp{*}(v) \ge 0  
\quad (v \in T).
\end{align}
For example, 
the arborescence algorithm of Edmonds~\cite{Edm67} finds an optimal solution $\rho\sp{*}$ such that 
$\rho\sp{*}(v) = \min\{w'(a) \colon  a = uv \}$ 
for every $v \in T$. 
Let $\rho\sp{*} \in \ZZ_+\sp{2\sp{T}}$ be an integral optimal solution for (D$'$) 
satisfying \eqref{EQrhovnonneg}. 
Also let $\tilde{B}_T$ be a minimum-weight $r_T$-arborescence in $D_T$
such that $R(\tilde{B}_T[T]) = \suppp(-\partial\xi|_{T})$
and $x'$ be the characteristic vector of this $\tilde{B}_T$;
cf.~Lemma~\ref{LEMarboRt}.

Similarly, on the $S$-side, 
we consider another directed graph $D_S= (V_S, A_S)$ 
with arc weight $w'' \in \ZZ\sp{A_S}$ defined as 
\begin{align*}
V_S = \{r_S\} \cup S, 
\quad
A_S = \{u r_S \colon u \in S\} \cup A[S], 
\quad
w''(uv) = 
\begin{cases}
-p(u)   & (v = r_S), \\
w(uv)  & (v \in S)
\end{cases}
\end{align*}
with a new vertex $r_S$.
We consider an arc subset such that the reversal of its arcs is an
$r_S$-arborescence.
Let $\tilde{B}_S$ be such an arc subset of minimum weight that satisfies
$R\sp{*}(\tilde{B}_S[S]) = \suppp(\partial\xi|_S)$.
Also let 
$\pi\sp{*} \in \ZZ_+\sp{2\sp{S}}$ 
be an integral optimal solution for the associated dual problem
satisfying $\pi\sp{*}(u) \ge 0$ for all $u \in S$.

Using $\pi\sp{*}$ and $\rho\sp{*}$ above
as well as $F = \{a \in A[S,T] \colon \xi(a) =1\}$,
define
$x\sp{*} \in \{0,1\}\sp{A}$, 
$y\sp{*}\in \ZZ\sp{2\sp{S}}$, and
$z\sp{*}\in \ZZ\sp{2\sp{T}}$ by 
\begin{align}
\label{EQxstar}
& x\sp{*} = \chi_{F \cup \tilde{B}_S[S] \cup \tilde{B}_T[T]},
 \\
\label{EQystar}
& y\sp{*}(S') = \pi\sp{*}(S')
\quad (\emptyset \not= S' \subseteq S),
 \\
\label{EQzstar}
& z\sp{*}(T') = \rho\sp{*}(T')
\quad(\emptyset \not= T' \subseteq T).
\end{align}
We prove that $x\sp{*}$ and $(y\sp{*},z\sp{*})$ 
are optimal solutions for (P) and (D), respectively.

\begin{lemma}
\label{lem:feasible}
$x\sp{*}$ and 
$(y\sp{*},z\sp{*})$ 
defined in \eqref{EQxstar}, \eqref{EQystar}, and \eqref{EQzstar}, respectively, 
are feasible for {\rm (P)} and {\rm (D)}, 
respectively. 
\end{lemma}

\begin{proof}
Since the arc set $F \cup \tilde{B}_S[S] \cup \tilde{B}_T[T]$ is a bibranching in $D=(V,A)$
by $R(\tilde{B}_T[T]) = \suppp(-\partial\xi|_{T})$ and
$R\sp{*}(\tilde{B}_S[S]) = \suppp(\partial\xi|_{S})$,
it is clear that $x\sp{*} = \chi_{F \cup \tilde{B}_S[S] \cup \tilde{B}_T[T]}$ is feasible for (P). 
As for $(y\sp{*}, z\sp{*})$, 
we first show that it satisfies \eqref{dual1}. 
For $a = uv \in A[S,T]$, 
by \eqref{EQoptpot2F} and Lemma~\ref{lem:tight}, we have that 
$w(a) + p(u) - q(v) \ge 0$, 
and hence 
\begin{align*}
\sum_{S' \colon a \in \delta\sp{+}S'}y\sp{*}(S') + 
\sum_{T' \colon a \in \delta\sp{-}T'}z\sp{*}(T') 
{}&{}= \sum_{S' \colon a \in \delta\sp{+}S'}\pi\sp{*}(S') + \sum_{T' \colon a \in \delta\sp{-}T'}\rho\sp{*}(T') \notag \\
{}&{}\le -p(u) + q(v) 
\\
{}&{}\le w(a) , \notag
\end{align*}
where
$\sum_{T' \colon a \in \delta\sp{-}T'}\rho\sp{*}(T') 
\le w'(a) = q(v)$ by \eqref{arbdual1} and 
the definition of $w'$,
and similarly
$\sum_{S' \colon a \in \delta\sp{+}S'}\pi\sp{*}(S') \le w''(a) = -p(u)$. 
For $a\in A[T]$, 
it follows from \eqref{arbdual1} that 
\begin{align*}
\sum_{S' \colon a \in \delta\sp{+}S'}y\sp{*}(S') + 
\sum_{T' \colon a \in \delta\sp{-}T'}z\sp{*}(T') = 
\sum_{T' \colon a \in \delta\sp{-}T'}\rho\sp{*}(T')  \le w'(a)=w(a). 
\end{align*}
The case of $a\in A[S]$ can be treated similarly.

Constraint~\eqref{dual3} is satisfied by 
\eqref{arbdual2} and \eqref{EQrhovnonneg}.
Similarly \eqref{dual2} is satisfied. 
\end{proof}

\begin{theorem}
$x\sp{*}$ and 
$(y\sp{*},z\sp{*})$ 
defined in \eqref{EQxstar}, \eqref{EQystar}, and \eqref{EQzstar}, respectively, 
are optimal solutions for {\rm (P)} and {\rm (D)}, respectively. 
\end{theorem}
\begin{proof}
By Lemma~\ref{lem:feasible}, 
it suffices to prove that 
$x\sp{*}$ and $(y\sp{*},z\sp{*})$ satisfy 
the complementary slackness conditions~\eqref{EQx}--\eqref{EQz}. 
To show \eqref{EQx},
assume $x\sp{*}(a)>0$.
For $a \in A[S,T]$, 
$x\sp{*}(a)>0$ means $x'(a)=\xi(a)=1$. 
Then, it follows from 
\eqref{EQarbx},
its counterpart for the $S$-side, and \eqref{EQtight}
that 
		\begin{align*}
		\sum_{S'\colon a \in \delta\sp{+}S'}y\sp{*}(S')
		+ \sum_{T'\colon a \in \delta\sp{-}T'}z\sp{*}(T') 
		=  w''(a) + w'(a) 
		= -p(u) + q(v)
		= w(a). 
		\end{align*} 
For $a \in A[T]$, 
$x'(a)=x\sp{*}(a) > 0$
implies that  
		$$\sum_{S'\colon a \in \delta\sp{+}S'}y\sp{*}(S')
		+ \sum_{T'\colon a \in \delta\sp{-}T'}z\sp{*}(T') 
		= \sum_{T'\colon a \in \delta\sp{-}T'}\rho\sp{*}(T') = w'(a) = w(a)$$
		by \eqref{EQarbx}. 
The case of $a \in A[S]$ can be treated similarly.

We next consider \eqref{EQz},
while noting that \eqref{EQy} can be shown similarly.
To show \eqref{EQz}, let $z\sp{*}(T') >0$,
where $\emptyset \not= T' \subseteq T$.
We are to show $x\sp{*}(\delta\sp{-}T')=1$.

If $|T'| \ge 2$, 
\eqref{EQarbrho} with $\rho\sp{*}(T') = z\sp{*}(T') > 0$ implies
		$|\tilde{B}_T \cap \delta\sp{-}T'|=1$ in $D_T$. 
		Denote the unique arc in  
		$\tilde{B}_T \cap \delta\sp{-}T'$ by $uv$.  
		For $v \in T \setminus \suppp (-\partial\xi|_{T})$, 
		it is clear that $x\sp{*}(\delta\sp{-}T') =1$,
	and hence \eqref{EQz} holds.
		For $v \in \suppp(-\partial\xi|_{T})$, 
		$x\sp{*}(\delta\sp{-}T')= -\partial\xi(v) \ge 2$ 
          would imply $q(v) =0$ by Lemma~\ref{lem:deg2},
whereas  
		$0 < z\sp{*}(T') = \rho\sp{*}(T') \le w'(a) = q(v)$;
		a contradiction. 
		Hence 
		$x\sp{*}(\delta\sp{-}T')= 1$ must hold.

		When $|T'|=1$,  we have $T' =\{v\}$ for some $v \in T$. 
		If 
		$v \not\in  \suppp(-\partial\xi|_{T})$, 
		then $x\sp{*}(\delta\sp{-}v) = 1$ holds since $\tilde{B}_T[T]$ is a branching in $D[T]$. 
		If $v \in \suppp(-\partial\xi|_{T})$, 
		then 
		again 
		$x\sp{*}(\delta\sp{-}T')= -\partial\xi(v) \ge 2$ would imply 
		$z\sp{*}(T') \le q(v)=0$ by Lemma~\ref{lem:deg2}, a contradiction.
		Hence 
		$x\sp{*}(\delta\sp{-}T')= 1$ must hold.
\end{proof}

\section*{Acknowledgement}
This work is supported by The Mitsubishi Foundation, 
CREST, JST, Grant Numbers JPMJCR14D2, JPMJCR1402, Japan, and 
JSPS KAKENHI Grant Numbers 16K16012, 26280001, 26280004.

\end{document}